\documentclass[11pt, oneside]{article}   	
\usepackage{amssymb}
\usepackage{amsthm}
\usepackage{amsmath}
\usepackage{float}
\usepackage{makeidx}
\usepackage{amsfonts}
\usepackage[margin=2.5cm]{geometry}
\usepackage{verbatim}
\usepackage{tikz}
\usetikzlibrary{arrows}

\usepackage{enumerate}
\theoremstyle{plain}
\newtheorem{theorem}{Theorem}[section]

\newtheorem{lemma}[theorem]{Lemma}
\newtheorem{corollary}[theorem]{Corollary}
\newtheorem{conjecture}[theorem]{Conjecture}
\newtheorem{proposition}[theorem]{Proposition}

\theoremstyle{definition}

\numberwithin{equation}{section}
\usepackage{graphicx}

\begin{document}
\date{\today}
\title{On star edge colorings of bipartite and subcubic graphs}

\author{
Carl Johan Casselgren\footnote{Department of Mathematics, 
Link\"oping University, 
SE-581 83 Link\"oping, Sweden.
{\it E-mail address:} carl.johan.casselgren@liu.se}
\and
Jonas B. Granholm\footnote{Department of Mathematics, 
Link\"oping University, 
SE-581 83 Link\"oping, Sweden.
{\it E-mail address:} jonas.granholm@liu.se}
\and
Andr\'e Raspaud\footnote{LaBRI,
Bordeaux University,
350, cours de la Lib\'eration,
33405 Talence cedex,
France.
{\it E-mail address:} raspaud@labri.fr}
}
\maketitle

\bigskip
\noindent
{\bf Abstract.}
A star edge coloring of a graph is a proper edge coloring with no
$2$-colored path or cycle of length four.
The star chromatic index $\chi'_{st}(G)$ of $G$ is the minimum
number $t$ for which $G$ has a star edge coloring with $t$ colors.
We prove upper bounds for the star chromatic index of 
bipartite graphs $G$ where all vertices in one part
have maximum degree $2$ and all vertices in the other part has maximum
degree $b$. Let $k$ be an integer ($k\geq 1$), we prove that if $b=2k+1$ then $\chi'_{st}(G) \leq 3k+2$;
and if $b=2k$, then $\chi'_{st}(G) \leq 3k$; both upper bounds are sharp.
We also consider complete
bipartite graphs; in particular we determine
the star chromatic index of such graphs when one part has size at most $3$,
and prove upper bounds for the general case.

Finally, we consider the well-known conjecture that subcubic graphs have
star chromatic index at most $6$; in particular we settle this conjecture
for cubic Halin graphs.

\bigskip


\section{Introduction}

    A {\em star edge coloring} of a graph is a proper edge coloring with no
    $2$-colored path or cycle of length four.
    The {\em star chromatic index} $\chi'_{st}(G)$ of $G$ is the minium
    number $t$ for which $G$ has a star edge coloring with $t$ colors.

    Star edge coloring was recently introduced by Liu and Deng \cite{LiuDeng},
    motivated by the vertex coloring version, see e.g. 
    \cite{Albertson, FertinRaspaudReed}.
    This notion is intermediate between {\em acyclic edge coloring},
    where every two-colored subgraph must be acyclic, and 
    {\em strong edge coloring}, where every color class
    is an induced matching.

    Dvorak et al.~\cite{DvorakMoharSamal} studied star edge colorings
    of complete graphs and obtained the currently best upper and lower
    bounds for the star chromatic index of such graphs. A fundamental
    open question here is to determine whether $\chi'_{st}(K_n)$ is linear
    in $n$.
    Bezegova et al.~\cite{BezegovaLuzar} investigated star edge colorings
    of trees and outerplanar graphs. 
    Wang et al. \cite{WangWangWang, WangWangWang2} quite recently
    obtained some upper bounds on the star chromatic index of
    graphs with maximum degree four, and also for some families of planar 
    and related classes of graphs.
    Besides these results,
    very little is known about star edge colorings.

    In this paper, we primarily consider star edge colorings
    of bipartite graphs. 
    As for complete graphs,
    a fundamental problem for complete bipartite graphs is to determine
    whether the star chromatic index is a linear function on the
    number of vertices.
    We determine the star chromatic index of complete
    bipartite graphs where one part has size at most $3$,
    and obtain some bounds on the star chromatic index for
    larger complete bipartite graphs. Note that the complete
    bipartite graph $K_{r,s}$ requires exactly $rs$ colors
    for a strong edge coloring; indeed it has been conjectured
    \cite{BrualdiMassey}
    that any bipartite graph where the parts have
    maximum degrees $\Delta_1$ and $\Delta_2$, respectively, has
    a strong edge coloring with $\Delta_1 \Delta_2$ colors.
    As we shall see, for star
    edge colorings the situation is quite different.

    Furthermore, we study star chromatic index of bipartite graphs
    where the vertices in one part all have small degrees.
    Nakprasit \cite{Nakprasit} proved that if $G$ is a bipartite graph
    where the maximum degree of one part is $2$, then $G$ has a
    strong edge coloring with $2 \Delta(G)$ colors.
    Here we obtain analogous results for star edge colorings:
    we obtain a sharp upper bound for the star chromatic
    index of a
    bipartite graph where one part has maximum degree two.

Finally, we consider the following conjecture first posed in 
\cite{DvorakMoharSamal}.

\begin{conjecture}
\label{conj:cubic}
    If $G$ has maximum degree at most $3$, then $\chi'_{st}(G) \leq 6$.
\end{conjecture}

Dvorak et al. \cite{DvorakMoharSamal} proved a slightly weaker version of Conjecture
\ref{conj:cubic}, namely that
$\chi'_{st}(G) \leq 7$ if $G$ is subcubic;
Bezegova et al established \cite{BezegovaLuzar} that
Conjecture \ref{conj:cubic} holds for all trees and outerplanar graphs,
while it is still open for e.g. planar graphs.
In this paper we verify that the conjecture holds for 
some families of graphs with maximum degree three, namely
bipartite graphs where one part has maximum degree $2$, cubic Halin graphs
and another family of planar graphs.

\section{Bipartite graphs}

    In this section we consider star edge colorings of bipartite graphs.
    We first consider complete bipartite graphs.
    Trivially $\chi'_{st}(K_{1,d}) = d$, 

    It is straightforward that $\chi'_{st}(K_{2,2})=3$.
    For general complete bipartite graphs where
    one part has size $2$, we have the following easy observation.

\begin{proposition}
\label{prop:K2d}
	For the complete bipartite graph $K_{2,d}$, $d \geq 3$, we have 
	$\chi'_{st}(K_{2,d}) = 2d - \left\lfloor \frac{d}{2} \right\rfloor$.
\end{proposition}
\begin{proof}
	Suppose $K_{2,d}$ has parts $X$ and $Y$,
	where $X=\{x_1, x_2\}$ and $Y=\{y_1,\dots, y_d\}$

	If $x_1$ and $x_2$ have at least
	$\lfloor d/2 \rfloor +1$ common colors on their incident edges, 
	say $1,\dots, \lfloor d/2 \rfloor +1$,
	then
	there is at least one vertex in $Y$ which is incident with two edges
	both of which have colors from $\{1,\dots,\lfloor d/2 \rfloor +1\}$; this implies
	that there is a $2$-colored $P_4$ or $C_4$ in $K_{2,d}$.
	Hence, there are at least $2d- \lfloor d/2 \rfloor$ distinct colors
	in a star edge coloring of $K_{2,d}$.
	
	To prove the upper bound, we give an explicit star edge coloring $f$
	of $K_{2,d}$. 
	We set $f(x_1y_i)= i$, for $i=1,\dots, d$, and
	$$f(x_2y_i) = d-i+1, \quad i=1,\dots, \lfloor d/2 \rfloor, \text{ and }
	f(x_2y_i) = i+\lfloor (d+1)/2 \rfloor, \quad i=\lfloor d/2 \rfloor+1,\dots, d.$$
	The coloring $f$  is a star edge coloring using exactly 
	$2d- \left\lfloor \frac{d}{2} \right\rfloor$ colors.
	
\end{proof}
 
    Wang et al.~\cite{WangWangWang} proved that $\chi'_{st}(K_{3,4}) =7$,
		and it is known that $\chi'_{st}(K_{3,3}) =6$.
    Using Proposition \ref{prop:K2d}, we can prove the following. 
		For an edge coloring $f$ of a graph $G$ and a vertex $u$ of $G$, 
		we
		denote by $f(u)$ the set of colors of all edges incident with $u$.
    	
	\begin{theorem}
	\label{th:3d}
	For the complete bipartite graph $K_{3,d}$, $d \geq 5$, it holds
	that $\chi'_{st}(K_{3,d}) = 3 \left\lceil\frac{d}{2}\right\rceil$.
	\end{theorem}
	\begin{proof}
	    Let us first consider the case when $d$ is even.
	    The lower bound $\frac{3d}{2}$ follows immediately from
	    Proposition \ref{prop:K2d}, so let us turn to the proof
	    of the upper bound. We shall give an explicit star edge coloring
	    of $K_{2,d}$ using $\frac{3d}{2}$ colors.
	    
	    Let $X$ and $Y$ be the parts of $K_{3,d}$,
	    where $X= \{x_1, x_2, x_3\}$, $Y = U \cup V$,
	    $U = \{u_1,\dots, u_k\}$, $V=\{ v_1, \dots, v_k\}$, and $d=2k$.
	    We define a star edge coloring $f$ by setting
	    
	    $$f(x_1u_i)= i,  \quad
	    f(x_2u_i)= i+k, \text{ and }
        f(x_3u_i)= i+2k, \quad i=1,\dots,k,$$
	   and 
	   \begin{itemize}
	       \item $f(x_1v_i) = i+k$, $i=1,\dots, k$, 
	       
	       \item $f(x_2v_i)= i+2k+1$, $i=1,\dots, k-1$, and $f(x_2v_k)=2k+1$,
	       
	       \item $f(x_3v_i) = i+2$, $i=1,\dots, k-2$, $f(x_3v_{k-1})=1$, and
	       $f(x_3v_k)= 2$.
	   \end{itemize}
	   Clearly, $f$ is a proper edge coloring of $K_{3,2k}$ with $3k$ colors.
	   Suppose that $K_{3,2k}$ contains a $2$-edge-colored path or cycle
	   $F$ with four edges. Let us prove that 
	   $F$ does not contain two edges $e_1$
	   and $e_2$ incident with the same vertex from $X$, say $x_1$,
	   and two other edges $e_3$ and $e_4$ incident with another vertex
	   from $X$, say $x_2$. Then, since the restriction $f'$ of $f$
	   to the subgraph induced by $X \cup U$ satisfies
	   $f'(x_1) \cap f'(x_2) = \emptyset$ (and similarly for the subgraph
			induced by $X \cup V$), we
	   may assume that $f(e_1) \in \{1,\dots, k\}$ and
	   $f(e_2) \in \{k+1, \dots, 2k\}$. However, no edge incident
	   with $x_2$ is colored by a color from $\{1,\dots, k\}$,
	   which contradicts that $F$ is $2$-edge-colored.
	   
	   Suppose now that there is a $2$-edge-colored path $F$ on $4$ edges, where exactly
	   two edges are incident to the same vertex from $X$, say $x_2$.
	   Then, as before, we may assume that
	   the two edges $e_2$ and $e_3$ of $F$ that are
	   incident with $v_2$ satisfy that $f(e_2) \in \{2k+1, \dots, 3k\}$
	   and $f(e_3) \in \{k+1,\dots 2k\}$. This means that the edge $e_1$
	   of $F$ that is incident with $x_1$ must satisfy $f(e_1) \in \{k+1, \dots, 2k\}$,
	   and so, $f(e_1) = f(e_3)$. 
	   Thus, for the edge $e_4$ of $F$ incident with $x_3$ it holds that
	   $f(e_4) \in \{2k+1,\dots, 3k\}$ and $f(e_4) = f(e_2)$.
	   However, by the construction of $f$ we have that $f(e_4) = f(e_2)-1$
	   (except if $f(e_2) =2k+1$, which implies that $f(e_4) = 3k$);
	   this contradicts that $F$ is $2$-edge-colored.
	   
	   Let us now consider the case when $d$ is odd; suppose $d = 2k+1$.
	   The upper bound follows immediately from the even case, since
	   $K_{3, 2k+1}$ is a subgraph of $K_{3, 2k+2}$. Let us prove
	   the lower bound.
	   
	   Let $X$ and $Y$ be the parts of $K_{3,d}$, where $X=\{x_1, x_2, x_3\}$ and
	   consider a star edge coloring of $K_{3,d}$.
        Since
        the subgraph induced by $\{x_1,x_2\} \cup Y$ is isomorphic to
        $K_{2,d}$, there are at most $k$ colors which can appear
        at both $x_1$ and $x_2$. Since the same holds for $x_1$ and $x_3$,
        and $x_2$ and $x_3$, it follows that we need at least 
        $3k+3$ colors for a star edge coloring of $K_{3,d}$.
	\end{proof}
	
	Next, we consider complete bipartite graphs where the parts
	have size at least $4$.
	Let us first establish a lower bound on the star chromatic index
	of such graphs.
	
	\begin{proposition}
	\label{prop:K4d}
	For the complete bipartite graph $K_{4,d}$ ($d\geq4$) it holds that
	$\chi'_{st}(K_{4,d})\ge \frac{10d}{6}$.
	\end{proposition}
	\begin{proof}
	    Let $X$ and $Y$ be the parts of $K_{4,d}$, where $X=\{x_1,x_2,x_3,x_4\}$,
	    and let $C_i$ be the set of colors used on the edges
			incident with the vertex $x_i$.
	    From the argument in Proposition~\ref{prop:K2d},
	    we can conclude that the sets $C_i$ may overlap in at most $d/2$ colors.
	    Thus the optimization problem
	    \begin{align*}
	        \text{minimize } \biggl|\bigcup_{i=1}^4 C_i\biggr|\\
	        \text{subject to } |C_i|&=d&i&=1,2,3,4\\
	        |C_i\cap C_j|&\le\frac{d}{2}&i,j&=1,2,3,4,\ i\ne j
	    \end{align*}
	    gives a lower bound on $\chi'_{st}(K_{4,d})$.
	    A linear integer program description of this problem is given 
			in Appendix~\ref{app:opt}.
	    Solving the linear relaxation of this problem gives the desired lower bound.
	\end{proof}
	
	For $d\leq12$, the values of $\chi'_{st}(K_{4,d})$ is given in Table~\ref{tab:K4d}.
	Explicit colorings realizing these values are given in Appendix~\ref{app:colorings},
	and for all values of $d$ except $d=6$ they can be proven optimal
	by solving the optimization problem in the proof of Proposition~\ref{prop:K4d}.
	For the case $d=6$, a computer search showing that
	the value in Table~\ref{tab:K4d} is optimal has been conducted.
	
	\begin{table} [H]
	    \centering
	    \begin{tabular}{c|c}
	       $d$ & $\chi'_{st}(K_{4,d})$ \\
	       \hline
	        4 &  7 \\
	        5 & 10 \\
	        6 & 11 \\
	        7 & 13 \\
	        8 & 14 \\
	        9 & 16 \\
	       10 & 17 \\
	       11 & 20 \\
	       12 & 20
	    \end{tabular}

	    \caption{The values of $\chi'_{st}(K_{4,d})$ for $d\leq 12$.}
	    \label{tab:K4d}
	\end{table}

    The fact that
	$\chi'_{st}(K_{4,12})=20$
	can be used to derive a general upper bound on the star chromatic
	index of complete bipartite graphs with four vertices in one part.
	
	\begin{proposition}
	For the complete bipartite graph $K_{4,d}$ it holds that
	$\chi'_{st}(K_{4,d})\le20\left\lceil\frac{d}{12}\right\rceil$.
	\end{proposition}
	\begin{proof}
	    Set $s = \left\lfloor \frac{d}{12} \right\rfloor$.
	    The proposition follows easily by decomposing
	    $K_{4,d}$ into $s$ copies
	    $H_1,\dots, H_s$ of $K_{4,12}$
	    and possibly and one copy $J$
	    of a complete bipartite graph $K_{4,a}$,
	    where $1\leq a \leq 11$, and then using
	    disjoint sets of colors for star edge colorings of each of the
	    graphs $H_1,\dots, H_s$ and $J$;
	    these star edge colorings together
	    form a star edge coloring of
	    $K_{4,d}$.
	\end{proof}

	Note that it follows from Proposition \ref{prop:K4d} that the upper bound
	in the preceding proposition is in fact sharp for an infinite
	number of values of~$d$.
	
	\bigskip

	Using computer searches we have also determined
	the star chromatic index for some additional 
	complete bipartite graphs;
	see Tables~\ref{tab:K5d} and~\ref{tab:Krd}.
	Again, 
	explicit colorings appear in
	Appendix~\ref{app:colorings}.
	\begin{table}[H]
	    \centering
	    \begin{tabular}{c|c}
	       $d$ & $\chi'_{st}(K_{5,d})$ \\
	       \hline
	        5 & 11 \\
	        6 & 12 \\
	        7 & 14 \\
	        8 & 15 \\
	        9 & 17 \\
	       10 & 18 \\
	       11 & 20
	    \end{tabular}
	    \caption{The values of $\chi'_{st}(K_{5,d})$ for $d\le11$.}
	    \label{tab:K5d}
	\end{table}
	
	\begin{table} [H]
	    \centering
	    \begin{tabular}{cc|c}
	       $r$ & $d$ & $\chi'_{st}(K_{r,d})$ \\
	       \hline
	       6 &  6 & 13 \\
	         &  7 & 14 \\
	         &  8 & 15 \\
	       \hline
	       7 &  7 & 14 \\
	         &  8 & 15 \\
	       \hline
	       8 &  8 & 15
	    \end{tabular}
	    \caption{The values of $\chi'_{st}(K_{r,d})$ for $r=6,7,8$ and $d\le8$.}
	    \label{tab:Krd}
	\end{table}
	
	Moreover, using the idea in the proof of Proposition~\ref{prop:K4d},
	one can prove lower bounds on the star chromatic index
	for further families of complete bipartite graphs. Let
	us here just list a few cases corresponding to the values
	in the tables above:

	\begin{itemize}
	
	   \item  $\chi'_{st}(K_{5,d})\ge \frac{10d}{6}$ if $d\geq 5$;
    
        \item $\chi'_{st}(K_{6,d})\ge \frac{7d}{4}$ if $d \geq 6$;
        
	    \item $\chi'_{st}(K_{7,d})\ge \frac{7d}{4}$ if $d \geq 7$;
	    
	    \item $\chi'_{st}(K_{8,d})\ge \frac{18d}{10}$ if $d\ge8$.
	
	\end{itemize}

	Finally, let us note some further consequences 
	of the above results for
	general complete bipartite graphs. 
	By decomposing a general
	complete bipartite graph $K_{r,d}$ into complete bipartite graphs
	where one part has size e.g. at most $3$, 
	and using disjoint sets of colors for star edge colorings of
	distinct complete bipartite subgraphs,
	we deduce, using Theorem \ref{th:3d}, that
	$\chi'_{st}(K_{r,d}) \leq 3 \left\lceil\frac{d}{2}\right\rceil
	\left\lceil\frac{r}{3}\right\rceil$. 
	However, using the values of star chromatic indices in
	the Tables~\ref{tab:K5d} and~\ref{tab:Krd}, it is possible
	to deduce an upper bound on $\chi'_{st}(K_{r,d})$
	which is better for large values of $r$ and $d$.
	
	\begin{corollary}
	\label{cor:Krd}
	    For any $r,d \geq 1$ it holds that
	    $\chi'_{st}(K_{r,d}) \leq 15 \left\lceil\frac{d}{8}\right\rceil
	\left\lceil\frac{r}{8}\right\rceil$.
	\end{corollary}
	
	
	Note that Dvorak et al. 
	\cite{DvorakMoharSamal} obtained an asymptotically better bound:
	it follows from their results that for every $\varepsilon > 0$,
	there is a constant $C>0$, such that for every $n \geq 1$,
	$\chi'_{st}(K_{n,n}) \leq C n^{1+\varepsilon}$.
	
	\bigskip
	\bigskip
	
    Next, we turn to general bipartite graphs with restrictions
    on the vertex degrees. Our first task is to generalize
    Proposition \ref{prop:K2d} to general bipartite graphs
    with even maximum degree.
	
    In the following we use the notation $G=(X,Y;E)$ for a bipartite
    graph $G$ with parts $X$ and $Y$ and edge set $E=E(G)$. 
		We denote by $\Delta(X)$ and
    $\Delta(Y)$ the maximum degrees of the vertices in the parts $X$
    and $Y$, respectively. A bipartite graph $G=(X,Y;E)$ where
    all vertices in $X$ have degree $2$ and all vertices in $Y$
    have degree $d$ is called {\em $(2,d)$-biregular}.
		
	If the vertices in one part of a bipartite graph $G$ 
	has maximum degree $1$, then trivially
    $\chi'_{st}(G)= \Delta(G)$. For the case when the vertices in 
    one of the parts have maximum degree two, we have the following.
	
	\begin{theorem}
	\label{th:bireg}
		If $G = (X,Y;E)$ is a bipartite graph, where $\Delta(X) = 2$
		and $\Delta(Y)=2k$, then
		$\chi'_{st}(G) \leq 3k$.
	\end{theorem}
	
	Note that the upper bound in Theorem \ref{th:bireg} is sharp, as follows
	from Proposition \ref{prop:K2d}.

	For the proof of this theorem we shall use the following lemma.
	
	\begin{lemma}
		If $G$ is $(2,2k)$-biregular with parts $X$ and $Y$, then it decomposes
		into subgraphs $F_i$ such that $d_{F_i}(x) \in \{0,2\}$ for every $x \in X$
		and $d_{F_i}(y) = 2$ for every $y \in Y$.
	\end{lemma}
	\begin{proof}
		From a $(2,2k)$-biregular graph $G$, construct a $2k$-regular
		multigraph $H$ by replacing every path of length $2$ with an internal
		vertex of degre $2$ by a single edge. By Petersen's $2$-factor theorem
		\cite{Petersen},
		$H$ has a decomposition into $2$-factors; these $2$-factors
		induce the required subgraphs of $G$.
	\end{proof}
	
	\begin{proof}[Proof of Theorem \ref{th:bireg}]
		If $G=(X,Y;E)$ is not $(2,2k)$-biregular then it is a subgraph of such
		a graph, so it suffices to consider the case when $G$
		is $(2,2k)$-biregular.
		
		Assume, consequently, that $G$ is a $(2,2k)$-biregular graph.
		By the preceding lemma, $G$ decomposes into subgraph $F_1, \dots, F_k$
		such that $d_{F_i}(x) \in \{0,2\}$ for every $x \in X$
		and $d_{F_i}(y) = 2$ for every $y \in Y$.
		Since each $F_i$ is a collection of even cycles, it has
		a star edge coloring with three colors.
		
		For $i=1,\dots, k$, we color each $F_i$ with colors $3i-2, 3i-1, 3i$
		so that each $F_i$ gets a star edge coloring with $3$ colors.
		This yields a star edge coloring $f$
		of $G$; indeed, for every vertex $x$ of $X$, all colors on edges
		incident with $X$ is in $\{3i-2, 3i-1, 3i\}$ for some $i$;
		hence, since any (possible) $2$-colored cycle or path $J$ with four edges 
		contains at least two vertices from $X$, it must
		be colored
		by two colors from $\{3i-2, 3i-1, 3i\}$ for some $i$, which means
		that all edges of $J$ is in $F_i$. This contradicts that the
		restriction of $f$ to each $F_i$ is a star edge coloring.
		We conclude that $f$ is in fact a star edge coloring of $G$.
	\end{proof}

	For the case $d=3$, we can generalize Proposition \ref{prop:K2d}
	as follows.

	\begin{theorem}
	\label{th:bip23}
			If $G=(X,Y;E)$ is a bipartite graph with $\Delta(X)=2$ and $\Delta(Y)=3$,
			then $\chi'_{st}(G) \leq 5$.
	\end{theorem}
	
	To prove the theorem, we use the following well-known lemma, the proof of which
	is left to the reader, and the notion of a {\em list star edge coloring}.
	A {\em list assignment} $L$ for a graph $G$ is
	a map which assigns to each edge $e$ of $G$ a set  $L(e)$ of colors. 
	If each of the lists has size $k$, we call $L$ a 
	{\em $k$-list assignment}.
If $G$ admits a star edge coloring $\varphi$ such that 
$\varphi(e) \in L(e)$ for every edge $e$ of $G$, then $G$ is 
{\em star $L$-edge-colorable}; 
 $\varphi$ is a {\em star $L$-edge coloring of $G$}. The graph $G$ is
{\em star $k$-edge-choosable} if it is star $L$-edge-colorable for every 
list assignment $L$, where
$|L(e)| \geq k$ for every $e \in E(G)$.
	
\begin{lemma}
\label{lem:cycle}
		If $C$ is any cycle distinct from $C_5$, then
    $C$ is star $3$-edge-choosable.
\end{lemma}
	
	By the {\em distance} between two edges $e$ and $e'$ of a graph, 
	we mean the smallest number of edges in a path from an endpoint of $e$
	to an endpoint of $e'$.
	
	Before proving Theorem \ref{th:bip23}, we have to notice that 
	for $C_4$-free graphs satisfying the condition in Theorem \ref{th:bip23},
	this result can be deduced from the result of \cite{Ma05}. In \cite{Ma05}, 
	it is proved that the {\em incidence chromatic number} of a subcubic graph is at 
	most $5$; the {\em incidence chromatic number} of a graph $G$ is equal 
	to the strong chromatic index of  $G^*$, where  $G^*$ is the graph obtained 
	from $G$ by subdividing each  edge of $G$. 
	When $G$ is a cubic graph (with no multiple edges),
	the graph $G^*$ is a $(2,3)$-biregular graph with no cycles of length four.
	Now, since the strong chromatic index of a graph is
	an upper bound for its  star chromatic index, the result follows.
	
	Since the result of \cite{Ma05} only applies to $C_4$-free 
	$(2,3)$-biregular graphs, and to have a self-contained paper,
	we give our short proof of Theorem \ref{th:bip23}.
	
	\begin{proof}[Proof of Theorem \ref{th:bip23}]
			Since any cycle of even length has a star $3$-edge coloring, we
			may assume that $G$ has maximum degree $3$.
			
			Assume that $G$ is a counterexample to the theorem which minimizes
			$|V(G)| + |E(G)|$. Then $G$ satisfies the following:
			
			\begin{itemize}
			
				 \item[(i)] $G$ is connected;
				
					\item[(ii)] $G$ does not contain any vertex of degree $1$;
					
					\item [(iii)] no two vertices of degree $3$ are adjacent;
					
					\item[(iv)] $G$ does not contain two vertices of degree $3$
					that are linked by a path $P$ of length at least four, where all internal
					vertices of $P$ have degree $2$. Thus any vertex
					of degree $2$ has two neighbors of degree $3$,
					so $G$ is $(2,3)$-biregular.
			
			\end{itemize}
		Statements (i)-(iii) are straightforward. To see (iv), assume
		that $P$ is such a path, and let $u_1u_2$ and $u_2u_3$ be two adjacent
		edges of $P$ where $u_1, u_2$ and $u_3$ all have degree $2$.
		By assumption $G-u_2$ has a star edge coloring with $5$ colors.
		Now we can color $u_1u_2$ with a color not appearing on an edge of distance
		at most $1$ from $u_1u_2$ in $G$; there are at most four such edges, so this
		is possible. Next, we can color $u_2u_3$ by a color not appearing on $u_1u_2$
		or on any edge at distance at most $1$ from $u_2u_3$ 
		in $G-u_1u_2$; there are at most four
		such edges, so we can pick a color from $\{1,2,3,4,5\}$ for $u_2u_3$. This
		yields a star $5$-edge coloring of $G$; a contradiction, and so, (iv) holds.

		Let $C_{2k}=  u_1 u_2 \dots u_{2k}u_1$ be a shortest cycle of $G$;
		if $G$ does not have a cycle, then it is star $5$-edge-colorable by
		a result of \cite{BezegovaLuzar}, a contradiction.
		Then the graph $H= G-E(C_{2k})$ is star $5$-edge colorable.
		
		We define a new star edge coloring $f$ of $H$ by recoloring 
		every pendant edge $e$ 
		of $H$ by a color from $\{1,2,3,4,5\}$ 
		not appearing on any edge of distance at most $1$ from 
		from $e$; there are at most three such edges, so this is possible.
		
		Next, we define a list assignment $L$ for $C_{2k}$ with colors from
		$\{1,2,3,4,5\}$ by
		for each edge $u_iu_{i+1}$ of $C_{2k}$ 
		forbidding the colors on the two pendant edges
		of $H$ with smallest distance to $u_iu_{i+1}$ in $G$. Then every
		edge of $C_{2k}$ receives a list of size at least $3$; so by Lemma
		\ref{lem:cycle} is has star $L$-edge coloring. This coloring
		along with the star edge coloring $f$ of $H$ form a star edge coloring of $G$
		with $5$ colors.
	\end{proof}
	
	Note that the preceding theorem settles a particular case of
	Conjecture \ref{conj:cubic}.

	Let us briefly remark that
	there are $(2,3)$-biregular graphs with $\chi'_{st}(G) =4$;
	while such examples with $\chi'_{st}(G) =3$ trivially 
	do not exist.
	Take two copies of $P_5$, and denote these copies by 
	$H_1 =u_1u_2u_3u_4u_5$ and $H_2= v_1v_2v_3v_4v_5$. 
	Next, we add the edges
	$E'=\{u_1v_2, u_5v_4, u_2v_1, u_4v_5\}$ to $H_1 \cup H_2$; the resulting
	graph $G$ is $(2,3)$-biregular. We define a proper edge coloring $f$
	of this graph by setting 
	$$f(u_1u_2) = f(u_3u_4) = f(v_3v_4) =3, \quad
	f(u_2u_3) = f(v_2v_3)= f(v_4v_5) =2,$$
	and
	$$f(u_4u_5)=f(v_1v_2)=1,$$
	and by coloring all edges in $E'$
	by color $4$. The coloring $f$ is a star edge coloring, because
	all edges of $E'$ are adjacent to edges of three distinct colors.
	
	\bigskip

	For bipartite graphs $G=(X,Y;E)$ with $\Delta(X)=2$ and odd maximum
	degree  at least $5$, we can use
	Theorem \ref{th:bip23} for proving the following.

	\begin{theorem}
	\label{th:biregOdd}
	    If $G=(X,Y;E)$ is a bipartite graph with $\Delta(X)=2$
	    and $\Delta(Y) =2k+1$, then
	    $\chi_{st}'(G) \leq 3k+2$.
	\end{theorem}
	
	Note that by Proposition \ref{prop:K2d}, Theorem \ref{th:biregOdd}
	is sharp.
	
	For the proof of Theorem \ref{th:biregOdd} we shall need the
	following theorem due to B\"abler, see e.g. \cite{HansonLotenToft}.
	
	\begin{theorem}
	\label{th:Babler}
		Let $G$ be a $(2k+1)$-regular multigraph. If $G$ has at most $2k$ bridges,
		then $G$ has a $2$-factor.
	\end{theorem}
	
	If the graph $G$ is obtained from $H$ by subdividing every edge of $H$,
	then we say that $H$ is the {\em condensed version} of $G$.
	Note that if $G$ is $(2,2k+1)$-biregular, then $H$ is $(2k+1)$-regular.
	
	The proof of Theorem \ref{th:biregOdd} is similar to the proof
	of the main result of \cite{HansonLotenToft}; hence we omit some details.

\begin{proof}[Proof of Theorem \ref{th:biregOdd}]
Since every graph satisfying the conditions in the theorem is a subgraph
of a $(2,2k+1)$-biregular graph, it suffices to prove the theorem
for $(2, 2k+1)$-biregular graphs.
The proof is by induction on $k$. The case $k=0$ is trivial, and Theorem
\ref{th:bip23} settles the case $k=1$.

Now assume that $k \geq 2$ and that $G$ is a $(2,2k+1)$-biregular graph.
Let $H$ be the condensed version of $G$; then $H$ is $(2k+1)$-regular.

If $H$ has at most one bridge, then by Theorem \ref{th:Babler}, $H$
has a $2$-factor. In $G$, this $2$-factor corresponds to a subgraph $F$,
where all vertices of $Y$ have degree $2$, and every vertex of $X$ has
degree $2$ or $0$. Note that the graph $G'$ obtained from
$G-E(F)$ by removing all isolated vertices is a $(2,2k-1)$-biregular
graph. By the induction hypothesis, $G'$ has a star edge coloring with
$3(k-1)+2$ colors. By star edge coloring all cycles of $F$ with
$3$ additional colors, we obtain a star edge coloring with $3k+2$
colors of $G$.

Now assume that $H$ has at least two bridges. We proceed as in 
\cite{HansonLotenToft}: Let $B_1, \dots, B_r$ be the maximal bridgeless connected
subgraphs obtained from $H$ by removing all bridges. For each subgraph
$B_i$ we construct a $(2k+1)$-regular multigraph containing $B_i$
by proceeding as follows: If 
there is an even number of bridges in $H$ with endpoints in $B_i$,
then we add a number of copies of the graph $A$ consisting of $2k$ parallell edges, 
the endpoints of which we join to endpoints in $B_i$ of removed bridges
by a single edge, respectively;
if there is an odd number of bridges with endpoints in $B_i$,
then we also add a subgraph $T$
consisting of a triangle $xyzx$ where $x$ and $y$ are 
joined by $k+1$ parallell edges, $x$ and $z$ are joined by $k$ parallell
edges, and $y$ and $z$ are joined by $k$ parallell edges, and $z$ is joined
by an edge to one endpoint  in $B_i$ of a removed bridge.
This yields a $(2k+1)$-regular multigraph $J_i$ containing $B_i$.
We set $J = J_1 \cup \dots \cup J_r$;
so $J$ is a $(2k+1)$-regular multigraph containing $B=B_1 \cup \dots \cup B_r$.
Note that in $J$ 
a bridge $b$ of $H$ is replaced by two edges joining the endpoints of $b$
with vertices of subgraphs isomorphic to $A$ or $T$.

By Theorem \ref{th:Babler}, $J$ has a $2$-factor. Thus, by proceeding
as in the preceding case, we may construct a star edge coloring $f$ with $3k+2$ colors
of the corresponding $(2,2k+1)$-biregular graph $D$ obtained from $J$
by subdividing all
edges of $J$.
Now, let $K$ be the graph obtained from $D$ by removing all
edges of $D$ that are in subgraphs that
correspond to the added subgraphs in $J$
that are isomorphic to $A$ or $T$,
and thereafter removing all isolated vertices.
The obtained graph $K$ is identical to the graph obtained 
from $B=B_1 \cup \dots\cup B_r$
by 
\begin{itemize}

\item[(i)] subdividing all edges of $B$, and 

\item[(ii)] for every bridge $uv$ of $H$
adding a path $uu_2u_3$ with origin at $u$, where $u_2, u_3$ are new vertices
and $u_3$ has degree $1$,
and adding a path $vv_2v_3$ with origin at $v$, where $v_2,v_3$ are new vertices
and $v_3$ has degree $1$.

\end{itemize}
Thus each path of length $2$ in $G$ that corresponds to a bridge in $H$ is
represented by two distinct paths of length $2$ in $K$; moreover, if we
identify every pair of such paths corresponding to the same bridge
in $H$, then we obtain a graph isomorphic to $G$.

Let $f_K$ be the restriction of $f$ to $K$; this is a star edge coloring
of $K$ with $3k+2$ colors. We recolor every pendant edge of $K$
by a color from $\{1,\dots, 3k+2\}$ which does not appear on an edge of distance
at most $1$ from the pendant edge; the obtained coloring $f'_K$ 
is a star edge coloring
with the property that no pendant edge is in a bicolored path of length
at least $3$.
Now, to obtain a star edge coloring of $G$ from $f'_K$
we may successively ``paste'' together components of $K$ by identifying paths
that correspond to the same
bridge in $H$ and permuting colors in one of the components 
so that the colorings agree on the identified paths;
this ``pasting process'' can be done exactly as in \cite{HansonLotenToft}
(e.g. by doing a Depth-First-Search in the tree with vertices for the subgraphs $B_i$
and edges for the bridges of $H$),
so we omit the exact details here.
Since in $K$, any bicolored path with a pendant edge has length at most $2$, 
this yields a star edge coloring of $G$.
\end{proof}

\section{Planar cubic graphs}

    As mentioned above, Conjecture \ref{conj:cubic} has been verified
    for outerplanar graphs.
    A particularly interesting special case of Conjecture \ref{conj:cubic}
    is planar graphs; this particular case is still wide open.
    In this section we provide two results 
    in this direction.

    Let us first prove that cubic Halin graphs
		have star chromatic index at most $6$.
    Note that this upper bound is sharp,
    since the complement of a $6$-cycle is a cubic Halin graph
		attaining this
    bound (see e.g. \cite{Luzar}).
    Our proof is similar to the proof in \cite{LihLiu} of the fact
    that cubic Halin graphs have strong chromatic index at most $7$.

\begin{theorem}
\label{th:cubicHalin}
	If $G$ is a cubic Halin graph, then $\chi'_{st}(G) \leq 6$.
\end{theorem}

\begin{proof}
	Let $G=T\cup C$, where $T$ is a tree and $C$ is an adjoint cycle containing
	all pendant vertices of $T$.
	Our proof proceeds by induction on the length $m$ of the cycle $C$.
	It is straightforward that every cubic Halin graph with $m \leq 5$
	has star chromatic index at most $6$, so let us assume that $m \geq 6$.

	Let $P = u_0 u_1 \dots u_l$ be a path of maximum length in $T$. Since
	$\Delta(T) \leq 3$ and $m \geq 6$, 
	$l \geq 5$. Moreover, since $P$ is maximum, all neighbors
	of $u_1$, except $u_2$, are leaves. We set $w=u_3, u=u_2, v= u_1$.
	Moreover, let $v_1$ and $v_2$ be the neighbors of $v$ on $C$ and
	label some other vertices in $G$ according to Figure~\ref{fig:Hhalin}.
	
	Since $d_G(u)=3$, there is a path $Q$ from $u$ to $x_1$ or $y_1$
	with $V(P) \cap V(Q) = \{u\}$. Suppose without loss of generality that
	there is such a path from $u$ to $y_1$. Then, since $P$ is a path
	of maximum length in $T$, $Q$ has length at most two; that is,
	$uy_3 \in E(T)$ or $u=y_3$. If the former holds, then $y_3y_2 \in E(T)$,
	and in the latter case, $uy_1 \in E(T)$.

	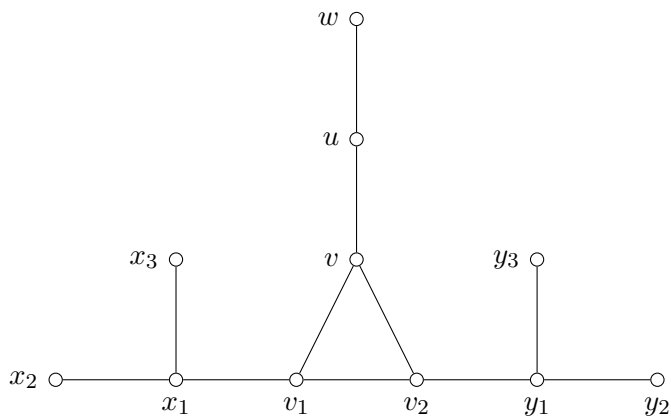
\begin{figure} [H]

	\begin{center}
	
	\begin{tikzpicture}[scale=0.8]
	\tikzset{vertex/.style = {shape=circle,draw,
												inner sep=0pt, minimum width=5pt}}
	\tikzset{edge/.style = {-,> = latex'}}
	
	\node[vertex] [label=left:$x_2$] (x_2) at  (-4, 0) {};
	\node[vertex] [label=below:$x_1$] (x_1) at  (-2, 0) {};
	\node[vertex] [label=left:$x_3$] (x_3) at  (-2, 2) {};
	\node[vertex] [label=below:$v_1$] (v_1) at  (0, 0) {};
	\node[vertex] [label=below:$v_2$] (v_2) at  (2, 0) {};
	\node[vertex] [label=left:$v$] (v) at  (1, 2) {};
	\node[vertex] [label=left:$u$] (u) at  (1, 4) {};
	\node[vertex] [label=left:$w$] (w) at  (1, 6) {};
	\node[vertex] [label=below:$y_1$] (y_1) at  (4, 0) {};
	\node[vertex] [label=below:$y_2$] (y_2) at  (6, 0) {};
	\node[vertex] [label=left:$y_3$] (y_3) at  (4, 2) {};
	
	\draw[edge] (x_2) to (x_1);
	\draw[edge] (x_1) to (x_3);
		\draw[edge] (x_1) to (v_1);
	\draw[edge] (v_1) to (v_2);
	\draw[edge] (v_1) to (v);
	\draw[edge] (v_2) to (v);
	\draw[edge] (u) to (v);
		\draw[edge] (u) to (w);
	\draw[edge] (v_2) to (y_1);
	\draw[edge] (y_2) to (y_1);
	\draw[edge] (y_3) to (y_1);

\end{tikzpicture}
\end{center}
\caption{The subgraph of $G$ used in the inducion step.}
	\label{fig:Hhalin}
\end{figure}

    \bigskip
	
	\pagebreak[2]
	\noindent
	{\bf Case 1.} $uy_3 \in E(T)$:
	
	\nopagebreak
	\medskip
	
	Let $z \in V(C)$ be the vertex distinct from $y_1$ 
	that is adjacent to $y_2$. Let $G'$ be the graph obtained from
	$G$ by removing vertices $v, v_1, v_2, y_1, y_2, y_3$, and adding
	two new edges $ux_1$ and $uz$. By the induction hypothesis,
	there is a star edge coloring $f'$ with colors $1,\dots, 6$ of $G'$.
	Without loss of generality we assume that $f'(uw) =1$, $f'(ux_1) = 2$
	and $f'(uz) =3$. Let $f'(x_1) =\{2, s_1, s_2\}$
	and $f'(z) = \{3, t_1, t_2\}$. Note that if $3 \in f'(x_1)$,
	then $2 \notin f'(z)$, and vice versa.
	From $f'$, we shall define a star $6$-edge coloring $f$ of $G$;
	we begin by setting $f(e) = f'(e)$ for all edges $e \in E(G') \cap E(G)$,
	$f(x_1v_1) = f(uv) = 2$ and $f(y_2z)= f(uy_3) = 3$. We extend $f$ to
	the remaining uncolored edges of $G$ by considering some different
	cases.
	
	\bigskip
    
    \noindent	
	{\bf Subcase 1.1.} $|\{1,\dots, 6\} \setminus \{2,3,s_1,s_2, t_1,t_2\}| \geq 2$:
	
	\medskip
	
	Let $\{ c_1, c_2 \} \subseteq \{1,\dots, 6\} \setminus \{2,3, s_1, s_2\}$.
	Without loss of generality, we assume that $c_2 \neq 1$
	and set $f(v_1v) = f(y_2y_3) = c_2$, and
	$f(v_1v_2) = f(y_1y_2) = c_1$. 
	To obtain a star edge coloring of $G$ we now properly color the edges
	$vv_2, v_2y_1, y_1y_3$ by two colors in 
	$\{1,\dots,6\}\setminus \{2,3,c_1,c_2\}$ so that neither of 
	$vv_2$ and $y_1y_3$ is colored $1$.
	
	\bigskip
	
	\noindent
	{\bf Subcase 1.2.} $|\{1,\dots, 6\} \setminus \{2,3,s_1,s_2, t_1,t_2\}| = 1$:
	
	\medskip
	
	Let
	$c_1 \in \{1,\dots,6\} \setminus \{2,3, s_1,s_2,t_1, t_2\}$.
	
	\bigskip
	
	\noindent
	{\bf Subcase 1.2.1} $\{s_1, s_2\} \cap \{t_1, t_2\} = \emptyset$:
	
	\medskip
	
	If $\{s_1, s_2\} \cap \{t_1, t_2\} = \emptyset$, then
	$2 \in f'(z)$ or $3 \in f'(x_1)$. Without loss of generality
	we assume that the former holds; so $t_1 = 2$ and thus $3\notin \{s_1,s_2\}$.
	
	If $c_1 = 1$, then we set $$f(v_1v_2) = f(y_1y_2)=c_1,
	f(vv_1) = f(y_1y_3)=t_2, f(vv_2)= f(y_2y_3) = s_1,
	\text{ and } f(v_2y_1)=s_2.$$
	This yields a star edge coloring of $G$.
	
	If, on the other hand $c_1 \neq 1$, then we set
	$f(vv_1) = f(y_2y_3)= c_1$ and $f(v_1v_2) = 3$.
	Without loss of generality, we further assume that $s_1 \neq 1$
	and set $f(vv_2) = f(y_1y_3)= s_1$, $f(v_2y_1)= t_2$ and
	$f(y_1y_2)=s_2$.

	\bigskip
	
	\noindent
	{\bf Subcase 1.2.2} $|\{s_1, s_2\} \cap \{t_1, t_2\}| = 1$:
	
	\medskip
	
	Suppose that $s_1 = t_1$ and note that the conditions imply that
	$2 \notin f'(z)$ and $3 \notin f'(x_1)$.
	
	If $c_1 = 1$, then we set $$f(v_1v_2) = f(y_1y_2)=c_1,
	f(vv_1) = f(y_1y_3) = t_2, f(vv_2) =f(y_2y_3)=s_2, \text{ and }
	f(v_2y_1) = s_1.$$
	
	If, on the other hand $c_1 \neq 1$, then we 
	set $f(vv_1) = f(y_2y_3)= c_1$, $f(v_1v_2) = 3$ and $f(y_1y_2)=2$.
	We then color the edges of the path $vv_2y_1y_3$ properly by colors
	$s_2$ and $t_2$ so that neither of $vv_2$ and $y_1y_3$ is colored $1$.
	
    \bigskip
    
    \noindent
    {\bf Subcase 1.3.} $\{1,\dots, 6\} \setminus \{2,3,s_1,s_2, t_1,t_2\} = \emptyset$:
    
    \medskip
    
    Without loss of generality, we assume that $s_1 = 1$.
    We obtain a star edge coloring of $G$ by setting
    $f(vv_1) = f(y_1y_3) = t_2$, $f(vv_2) = f(y_2y_3) = s_2$,
    $f(v_1v_2) = t_1$, $f(y_1y_2) = s_1$, and $f(y_1v_2) = 3$.
    
    \bigskip
	
	\noindent
	{\bf Case 2.} $u = y_3$:
	
	\nopagebreak
	\medskip

    Let $G'$ be the graph obtained from $G$ by removing
    $v, v_1, v_2, y_1$, and adding the new edges $ux_1$ and $uy_2$.
    By the induction hypothesis, there is a star $6$-edge coloring
    $f'$ of $G'$. 
    Without loss of generality we assume that $f'(uw) =1$, $f'(ux_1) = 2$
	and $f'(uy_2) =3$. Let $f'(x_1) =\{2, s_1, s_2\}$
	and $f'(y_2) = \{3, t_1, t_2\}$. 
	From $f'$, we shall define a star $6$-edge coloring $f$ of $G$;
	we begin by setting $f(e) = f'(e)$ for all edges $e \in E(G')\cap E(G)$,
	$f(x_1v_1) = 2$ and $f(y_1y_2) = 3$. We extend $f$ to
	the remaining uncolored edges of $G$ by considering some different
	cases.
	
	\bigskip
	
	\noindent
	{\bf Subcase 2.1.} $2 \notin \{t_1, t_2\}$:
	
	\medskip
	
	We set $f(uv)=3$ and $f(uy_1) = 2$ and consider three different subcases.
	
	\bigskip
	
	\noindent
	{\bf Subcase 2.1.1.} 
	$|\{1,\dots ,6\} \setminus \{2,3, s_1,s_2, t_1, t_2\}| \geq 2$:
	
	\medskip
	
	Let $\{c_1, c_2\} \subseteq \{1,\dots, 6\} \setminus \{2,3,s_1, s_2, t_1, t_2\}$.
	Without loss of generality, we assume that $c_2 \neq 1$
	and set $f(vv_1) = f(v_2y_1) = c_2$ and $f(v_1v_2)=c_1$.
	We then color $vv_2$ by a color from $\{1,\dots,6\} \setminus \{1,2,3,c_1,c_2\}$
	to obtain a star edge coloring of $G$.
	
	\bigskip
	
	\noindent
	{\bf Subcase 2.1.2.} 
	$|\{1,\dots ,6\} \setminus \{2,3, s_1,s_2, t_1, t_2\}| = 1$:
	
	\medskip
	
	Let
	$c_1 \in \{1,\dots,6\} \setminus \{2,3, t_1, t_2, s_2\}$.
	Suppose first that $3 \in \{s_1, s_2\}$, e.g. that $s_1=3$;
	then the colors $2, c_1, s_1,s_2, t_1, t_2$ are distinct 
	and we set $f(v_1v_2) = c_1$, $f(vv_1)=t_2$,
	$f(vv_2) = t_1$ and $f(v_2y_1)= s_2$
	to obtain a star edge coloring of $G$.

	Suppose now that $\{s_1,s_2\} \cap \{t_1,t_2\} \neq \emptyset$, e.g.
	$t_1 = s_1$; 
	then the colors $c_1, t_1, s_2, t_2$ are all distinct
	and not equal to $2$ or $3$. We
	set $f(vv_2) = t_1$, $f(v_2y_1) = s_2$, and
	color $vv_1, v_1v_2$ by the colors $c_1, t_2$
	so that $f(vv_1) \neq 1$.

	\bigskip
	
	\noindent
	{\bf Subcase 2.1.3.} 
	$\{1,\dots ,6\} \setminus \{2,3, s_1,s_2, t_1, t_2\} = \emptyset$:

    \medskip
    
    Without loss of generality, we assume that $s_2 \neq 1$,
    and set $$f(vv_1) = t_2, f(v_1v_2) = t_1, f(vv_2)=s_2, f(v_2y_1)=s_1.$$

	\bigskip
	
	\noindent
	{\bf Subcase 2.2.} $2 \in \{t_1, t_2\}$:
	
	\medskip
	
	We assume $t_1=2$; note that this implies that $3 \notin f'(x_1)$.
	We set $f(uy_1)= 2$, $f(vv_1)=3$ and
	color $uv$ by a color $c_1 \leq 6$ satisfying that
	$c_1 \notin f'(w) \cup \{2,3\}$.

	Assume first that $1 \notin \{s_1,s_2\}$. Then we set
	$f(v_1v_2)=1$, color $v_2y_1$ by a color
	$c_2 \in \{1,\dots, 6\} \setminus \{1,2,3,c_1,t_2\}$,
	and thereafter color $vv_2$ by a color from 
	$\{1,\dots, 6\} \setminus \{1,2,3,c_1,c_2\}$.
	
	Assume now that $1 \in \{s_1,s_2\}$, e.g. that $s_1=1$.
	If $c_1 \neq t_2$, then we color $v_2y_1$ by $c_1$,
	$v_1v_2$ by a color $c_2 \in \{1,\dots, 6\} \setminus \{1,2,3,c_1,s_2\}$,
	and $vv_2$ by a color $c_3 \in \{1,\dots, 6\} \setminus \{1,2,3,c_1,c_2\}$,
	
	If, on the other hand $c_1 =t_2$, then we color $v_1v_2$ by a color
	$c_2 \in \{1,\dots, 6\} \setminus \{1,2,3,c_1,s_2\}$, $vv_2$ by color $2$,
	and $v_2y_1$ by a color $c_3 \in \{1,\dots, 6\} \setminus \{1,2,3,c_1,c_2\}$
\end{proof}

  Finally, we have the following for planar graphs.

\begin{proposition}
\label{prop:girth}
    Let $G$ be a subcubic planar graph with girth at least $7$. If
    $G$ has a perfect matching, then $\chi'_{st}(G) \leq 6$.
\end{proposition}
\begin{proof}
    Let $G$ be a subcubic planar graph with girth at least $7$, and suppose
    that $M$ is a perfect matching in $G$. Let $H$
    be the graph obtained from $G$ by contracting all edges of $M$. Since $G$
    is planar and has girth at least $7$, $H$ is a planar triangle-free graph.
    Thus, by Gr\"otzsch's theorem, $H$ has a proper vertex coloring $\varphi$ 
    with colors $1,2,3$. We obtain a partial strong edge coloring $f$ of $G$
    by coloring every edge of $M$ by the color of the corresponding vertex
    in $H$. 
    
    Now, by Lemma \ref{lem:cycle}, $G-M$ has a star edge coloring
    with three colors; use colors $4,5,6$ for such a coloring $g$ of $G-M$.
    By combining the colorings $g$ and $f$, we obtain a star edge coloring
    of $G$ with colors $1,\dots,6$. Indeed, there is no $2$-colored
    path or cycle of length four with colors only in $\{1,2,3\}$
    or $\{4,5,6\}$, because both $f$ and $g$ are star edge colorings.
    Moreover, there is no $2$-colored path or cycle of length four
    with one color from $\{1,2,3\}$ and one color from $\{4,5,6\}$,
    because $f$ is a strong edge coloring of $M$ with respect to $G$.
\end{proof}

\section{Acknowledgement}
 Carl Johan Casselgren was supported by a grant from the Swedish
Research Council (2017-05077).\\

\noindent Andr\'e Raspaud was partially supported by the French ANR project HOSIGRA (ANR-17-CE40-0022).

\appendix

\section{The optimization problem in the proof of Proposition~\ref{prop:K4d}}
\label{app:opt}
The optimization problem in the proof of Proposition~\ref{prop:K4d}
can be formulated as a linear integer program in the following way.

\[\text{minimize } \sum_{i=0}^1 \sum_{j=0}^1 \sum_{k=0}^1 \sum_{\ell=0}^1 x_{ijk\ell}\]
subject to
\begin{align*}
\sum_{j=0}^1 \sum_{k=0}^1 \sum_{\ell=0}^1 x_{1jk\ell}&=d\\
\sum_{i=0}^1 \sum_{k=0}^1 \sum_{\ell=0}^1 x_{i1k\ell}&=d\\
\sum_{i=0}^1 \sum_{j=0}^1 \sum_{\ell=0}^1 x_{ij1\ell}&=d\\
\sum_{i=0}^1 \sum_{j=0}^1 \sum_{k=0}^1 x_{ijk1}&=d\\
\sum_{k=0}^1 \sum_{\ell=0}^1 x_{11k\ell}&=\frac{d}{2}\\
\sum_{j=0}^1 \sum_{\ell=0}^1 x_{1j1\ell}&=\frac{d}{2}\\
\sum_{j=0}^1 \sum_{k=0}^1 x_{1jk1}&=\frac{d}{2}\\
\sum_{i=0}^1 \sum_{\ell=0}^1 x_{i11\ell}&=\frac{d}{2}\\
\sum_{i=0}^1 \sum_{k=0}^1 x_{i1k1}&=\frac{d}{2}\\
\sum_{i=0}^1 \sum_{j=0}^1 x_{ij11}&=\frac{d}{2}\\
x_{ijk\ell}&\ge0\text{ and integer}\quad\forall i,j,k,\ell\in\{0,1\}
\end{align*}

\noindent
The optimal solution to the linear relaxation of this problem is $\tfrac{10d}{6}$,
which is attained by

\begin{gather*}
x_{0011}=
x_{0101}=
x_{0111}=
x_{0110}=
x_{1001}=
x_{1010}=
x_{1011}=
x_{1100}=
x_{1101}=
x_{1110}=\tfrac d6
\\
x_{0000}=
x_{0001}=
x_{0010}=
x_{0100}=
x_{1000}=
x_{1111}=0
\end{gather*}

\section{Star edge colorings of small complete bipartite graphs}
\label{app:colorings}

    In this appendix we give explicit colorings of some small complete
    bipartite graphs.
    The colorings are given in the form of an array where
	    rows and columns correspond to vertices, cells
	    correspond to edges and the contents of the cells correspond to
	    colors.
	
	\newcommand{\widecell}[1]{\ooalign{\hfil #1\hfil\cr\hphantom{99}}}

	\begin{figure}[H]
	    \centering
	    \begin{tabular}{|*{4}{c|}}
	     \hline
         \widecell 1 &  \widecell 2 &  \widecell 3 & \widecell 4\\
          \hline
         5 &  6 &  1 & 2\\
          \hline
         2 &  7 &  6 & 3\\
          \hline
         4 &  5 &  2 & 7\\
          \hline
	    \end{tabular}
	    \caption{A star edge coloring of $K_{4,4}$ with $7$ colors.}
	    \label{tab:K44}
	\end{figure}
         
	\begin{figure}[H]
	    \centering
	    \begin{tabular}{|*{5}{c|}}
	     \hline
         \widecell 1 &  \widecell 2 &  \widecell 3 &  4 & 5\\
          \hline
         6 &  7 &  1 &  2 & 8\\
          \hline
         2 &  8 &  4 &  9 & 10\\
          \hline
         4 &  5 &  7 & 10 & 6\\
          \hline
	    \end{tabular}
	    \caption{A star edge coloring of $K_{4,5}$ with $10$ colors.}
	    \label{tab:K45}
	\end{figure}
         
	\begin{figure}[H]
	    \centering
	    \begin{tabular}{|*{6}{c|}}
	    \hline
         \widecell 1 &  2 &  \widecell 3 &  4 &  5 & 6\\
          \hline
         7 &  8 &  9 &  1 &  2 & 3\\
          \hline
         2 &  9 &  4 & 10 &  6 & 11\\
          \hline
         5 & 11 &  2 &  8 & 10 & 7\\
          \hline
	    \end{tabular}
	    \caption{A star edge coloring of $K_{4,6}$ with $11$ colors.}
	    \label{tab:K46}
	\end{figure}
         
	\begin{figure}[H]
	    \centering
	    \begin{tabular}{|*{7}{c|}}
	    \hline
         \widecell 1 &  2 &  3 &  4 &  5 &  6 & 7\\
         \hline
         8 &  9 & 10 &  1 &  2 &  3 & 11\\
         \hline
         2 & 10 &  4 & 11 &  6 & 12 & 13\\
         \hline
         5 &  3 & 12 & 10 & 13 &  7 & 8\\
         \hline
	    \end{tabular}
	    \caption{A star edge coloring of $K_{4,7}$ with $13$ colors.}
	    \label{tab:K47}
	\end{figure}
         
	\begin{figure}[H]
	    \centering
	    \begin{tabular}{|*{8}{c|}}
	     \hline
         1 &  2 &  3 &  4 &  5 &  6 &  7 & 8\\
          \hline
         9 & 10 & 11 & 12 &  1 &  2 &  3 & 4\\
          \hline
         2 & 11 &  5 &  6 & 12 & 13 &  8 & 14\\
          \hline
        14 &  7 & 13 & 11 &  8 &  1 &  9 & 10\\
         \hline
	    \end{tabular}
	    \caption{A star edge coloring of $K_{4,8}$ with $14$ colors.}
	    \label{tab:K48}
	\end{figure}
         
	\begin{figure}[H]
	    \centering
	    \begin{tabular}{|*{9}{c|}}
	    \hline
         1 &  2 &  3 &  4 &  5 &  6 &  7 &  8 & 9\\
         \hline
        10 & 11 & 12 & 13 &  1 &  2 &  3 &  4 & 14\\
        \hline
         2 & 12 &  5 &  6 & 13 & 15 &  8 & 14 & 16\\
         \hline
         8 &  9 & 16 & 12 &  4 &  7 & 11 & 15 & 10\\
         \hline
	    \end{tabular}
	    \caption{A star edge coloring of $K_{4,9}$ with $16$ colors.}
	    \label{tab:K49}
	\end{figure}
         
	\begin{figure}[H]
	    \centering
	    \begin{tabular}{|*{10}{c|}}
	     \hline
         1 &  2 &  3 &  4 &  5 &  6 &  7 &  8 &  9 & 10\\
          \hline
        11 & 12 & 13 & 14 & 15 &  1 &  2 &  3 &  4 & 5\\
         \hline
         2 & 13 &  6 &  7 & 16 & 14 & 17 &  5 & 15 & 9\\
          \hline
         8 & 16 & 11 &  6 &  3 & 12 & 10 & 17 &  2 & 14\\
          \hline
	    \end{tabular}
	    \caption{A star edge coloring of $K_{4,10}$ with $17$ colors.}
	    \label{tab:K410}
	\end{figure}
         
	\begin{figure}[H]
	    \centering
	    \begin{tabular}{|*{11}{c|}}
	    \hline
         1 &  2 &  3 &  4 &  5 &  6 &  7 &  8 &  9 & 10 & 11\\
         \hline
        12 & 13 & 14 & 15 & 16 &  1 &  2 &  3 &  4 &  5 & 17\\
        \hline
         2 & 14 &  6 &  5 & 18 & 16 &  8 & 17 & 10 & 19 & 20\\
         \hline
         3 &  9 & 15 & 20 & 11 &  4 & 16 & 19 & 18 &  7 & 13\\
          \hline
	    \end{tabular}
	    \caption{A star edge coloring of $K_{4,11}$ with $20$ colors.}
	    \label{tab:K411}
	\end{figure}

	\begin{figure}[H]
	    \centering
	    \begin{tabular}{|*{12}{c|}}
	    \hline
	     1 &  2 &  3 &  4 &  5 &  6 &  7 &  8 &  9 & 10 & 11 & 12\\
	     \hline
		13 & 14 & 15 & 16 & 17 & 18 &  1 &  2 &  3 &  4 &  5 &  6\\
		\hline
		 2 & 15 &  7 &  8 & 10 & 19 & 17 & 20 &  6 & 18 & 16 & 11\\
		 \hline
		19 & 10 & 16 & 11 &  9 &  3 & 12 &  1 & 20 & 17 & 13 & 14 \\
		\hline
	    \end{tabular}
	    \caption{A star edge coloring of $K_{4,12}$ with $20$ colors.}
	    \label{tab:K412}
	\end{figure}

	\begin{figure}[H]
	    \centering
	    \begin{tabular}{|*{11}{c|}}
	    \hline
		 \widecell 1 &  2 &  3 &  \widecell 4 &  5\\
		 \hline
		 6 &  7 &  8 &  1 &  2\\
		 \hline
		 2 &  8 &  4 &  9 & 10\\
		 \hline
		 3 &  6 & 10 &  2 & 11\\
		 \hline
		 5 & 11 &  2 &  7 &  9\\
		 \hline
	    \end{tabular}
		\caption{A star edge coloring of $K_{5,5}$ with $11$ colors.}
	    \label{tab:K55}
	\end{figure}

	\begin{figure}[H]
	    \centering
	    \begin{tabular}{|*{12}{c|}}
	    \hline
		 \widecell 1 &  2 &  \widecell 3 &  4 &  5 &  6\\
		 \hline
		 7 &  8 &  9 &  1 &  2 &  3\\
		 \hline
		 2 &  9 &  4 & 10 &  6 & 11\\
		 \hline
		 5 & 11 &  2 &  8 & 10 &  7\\
		 \hline
		 8 & 10 &  1 & 12 &  4 &  2\\
		 \hline
	    \end{tabular}
		\caption{A star edge coloring of $K_{5,6}$ with $12$ colors.}
	    \label{tab:K56}
	\end{figure}

	\begin{figure}[H]
	    \centering
	    \begin{tabular}{|*{14}{c|}}
	    \hline
		 1 &  2 &  3 &  4 &  5 &  6 &  7\\
		 \hline
		 8 &  9 & 10 & 11 &  1 &  2 &  3\\
		 \hline
		 2 & 10 &  4 & 12 &  6 & 13 & 11\\
		 \hline
		 4 &  7 &  5 &  9 & 14 & 12 &  8\\
		 \hline
		11 &  3 & 13 &  5 & 10 &  7 & 14 \\
		\hline
	    \end{tabular}
		\caption{A star edge coloring of $K_{5,7}$ with $14$ colors.}
	    \label{tab:K57}
	\end{figure}

	\begin{figure}[H]
	    \centering
	    \begin{tabular}{|*{15}{c|}}
	    \hline
		 1 &  2 &  3 &  4 &  5 &  6 &  7 &  8\\
		 \hline
		 9 & 10 & 11 & 12 &  1 &  2 &  3 &  4\\
		 \hline
		 2 & 11 &  5 & 13 & 12 &  7 & 14 & 15\\
		 \hline
		 6 & 13 &  9 &  3 &  2 & 15 &  8 & 10\\
		 \hline
		13 &  7 &  1 & 10 & 14 &  4 &  9 &  5\\
		\hline
	    \end{tabular}
		\caption{A star edge coloring of $K_{5,8}$ with $15$ colors.}
	    \label{tab:K58}
	\end{figure}

	\begin{figure}[H]
	    \centering
	    \begin{tabular}{|*{17}{c|}}
	    \hline
		 1 &  2 &  3 &  4 &  5 &  6 &  7 &  8 &  9\\
		 \hline
		10 & 11 & 12 & 13 & 14 &  1 &  2 &  3 &  4\\
		\hline
		 2 & 12 &  5 &  6 & 13 & 15 & 16 & 17 &  8\\
		 \hline
		 5 &  9 &  7 & 10 & 12 & 11 & 17 &  4 & 15\\
		 \hline
		17 &  6 &  9 &  1 &  8 & 16 & 14 & 11 & 10\\
		\hline
	    \end{tabular}
		\caption{A star edge coloring of $K_{5,9}$ with $17$ colors.}
	    \label{tab:K59}
	\end{figure}

	\begin{figure}[H]
	    \centering
	    \begin{tabular}{|*{18}{c|}}
	    \hline
		 1 &  2 &  3 &  4 &  5 &  6 &  7 &  8 &  9 & 10\\
		 \hline
		11 & 12 & 13 & 14 & 15 &  1 &  2 &  3 &  4 &  5\\
		\hline
		 2 & 13 &  6 &  7 &  8 & 15 & 16 & 17 & 10 & 14\\
		 \hline
		 7 &  9 & 18 &  3 &  6 & 12 & 15 & 10 & 11 & 13\\
		 \hline
		 9 &  8 & 14 & 17 & 11 &  4 &  3 & 18 & 16 & 12\\
		 \hline
	    \end{tabular}
		\caption{A star edge coloring of $K_{5,10}$ with $18$ colors.}
	    \label{tab:K510}
	\end{figure}

	\begin{figure}[H]
	    \centering
	    \begin{tabular}{|*{20}{c|}}
	    \hline
		 1 &  2 &  3 &  4 &  5 &  6 &  7 &  8 &  9 & 10 & 11\\
		 \hline
		12 & 13 & 14 & 15 & 16 & 17 &  1 &  2 &  3 &  4 &  5\\
		\hline
		 2 & 14 &  6 &  5 & 18 & 19 &  8 & 20 & 16 & 17 &  9\\
		 \hline
		 6 & 18 & 11 & 19 &  7 & 13 & 15 & 16 &  4 &  8 & 12\\
		 \hline
		19 &  3 & 20 & 17 & 10 & 11 & 18 &  1 &  7 & 12 & 14\\
		\hline
	    \end{tabular}
		\caption{A star edge coloring of $K_{5,11}$ with $20$ colors.}
	    \label{tab:K511}
	\end{figure}

	\begin{figure}[H]
	    \centering
	    \begin{tabular}{|*{13}{c|}}
	    \hline
		 1 &  2 &  3 &  4 &  5 &  6\\
		 \hline
		 7 &  8 &  9 &  1 &  2 &  3\\
		 \hline
		 2 &  9 & 10 &  3 & 11 & 12\\
		 \hline
		 3 &  4 & 12 & 13 &  9 &  5\\
		 \hline
		 6 & 12 &  7 &  8 &  3 & 13\\
		 \hline
		10 &  3 & 13 &  6 &  1 & 11\\
		\hline
	    \end{tabular}
		\caption{A star edge coloring of $K_{6,6}$ with $13$ colors.}
	    \label{tab:K66}
	\end{figure}

	\begin{figure}[H]
	    \centering
	    \begin{tabular}{|*{14}{c|}}
	    \hline
		 \widecell 1 &  2 &  3 &  4 &  5 &  6 &  7\\
		 \hline
		 8 &  9 & 10 & 11 &  1 &  2 &  3\\
		 \hline
		 2 & 10 &  4 & 12 & 11 &  5 & 13\\
		 \hline
		 4 &  8 & 14 & 13 &  3 & 10 &  6\\
		 \hline
		 6 & 12 &  8 &  7 &  2 & 13 &  9\\
		 \hline
		 7 & 13 & 11 &  5 &  9 &  3 & 14\\
		 \hline
	    \end{tabular}
		\caption{A star edge coloring of $K_{6,7}$ with $14$ colors.}
	    \label{tab:K67}
	\end{figure}

	\begin{figure}[H]
	    \centering
	    \begin{tabular}{|*{15}{c|}}
	    \hline
		 \widecell 1 &  2 &  3 &  4 &  5 &  6 &  7 &  8\\
		 \hline
		 9 & 10 & 11 & 12 &  1 &  2 &  3 &  4\\
		 \hline
		 2 & 13 &  4 & 14 & 10 &  5 & 12 &  7\\
		 \hline
		 3 & 15 & 13 &  2 & 11 &  8 &  5 & 10\\
		 \hline
		 6 &  9 & 12 &  7 &  2 & 13 & 15 &  3\\
		 \hline
		 8 & 11 &  6 &  9 &  4 & 14 &  2 & 13\\
		 \hline
	    \end{tabular}
		\caption{A star edge coloring of $K_{6,8}$ with $15$ colors.}
	    \label{tab:K68}
	\end{figure}

	\begin{figure}[H]
	    \centering
	    \begin{tabular}{|*{14}{c|}}
	    \hline
		 1 &  2 &  3 &  4 &  5 &  6 &  7\\
		 \hline
		 8 &  9 & 10 & 11 &  1 &  2 &  3\\
		 \hline
		 2 & 10 &  4 & 12 & 11 &  5 & 13\\
		 \hline
		 4 &  8 & 14 & 13 &  3 & 10 &  6\\
		 \hline
		 6 & 12 &  8 &  7 &  2 & 13 &  9\\
		 \hline
		 7 & 13 & 11 &  5 &  9 &  3 & 14\\
		 \hline
		12 & 11 &  1 & 14 &  7 &  4 &  8\\
		\hline
	    \end{tabular}
		\caption{A star edge coloring of $K_{7,7}$ with $14$ colors.}
	    \label{tab:K77}
	\end{figure}

	\begin{figure}[H]
	    \centering
	    \begin{tabular}{|*{15}{c|}}
	    \hline
		 1 &  2 &  3 &  4 &  5 &  6 &  7 &  8\\
		 \hline
		 9 & 10 & 11 & 12 &  1 &  2 &  3 &  4\\
		 \hline
		 2 & 13 &  4 & 14 & 10 &  5 & 12 &  7\\
		 \hline
		 3 & 15 & 13 &  2 & 11 &  8 &  5 & 10\\
		 \hline
		 6 &  9 & 12 &  7 &  2 & 13 & 15 &  3\\
		 \hline
		 8 & 11 &  6 &  9 &  4 & 14 &  2 & 13\\
		 \hline
		11 & 12 &  5 &  6 & 14 & 15 &  1 &  2\\
		\hline
	    \end{tabular}
		\caption{A star edge coloring of $K_{7,8}$ with $15$ colors.}
	    \label{tab:K78}
	\end{figure}

	\begin{figure}[H]
	    \centering
	    \begin{tabular}{|*{15}{c|}}
	    \hline
		 1 &  2 &  3 &  4 &  5 &  6 &  7 &  8\\
		 \hline
		 9 & 10 & 11 & 12 &  1 &  2 &  3 &  4\\
		 \hline
		 2 & 13 &  4 & 14 & 10 &  5 & 12 &  7\\
		 \hline
		 3 & 15 & 13 &  2 & 11 &  8 &  5 & 10\\
		 \hline
		 6 &  9 & 12 &  7 &  2 & 13 & 15 &  3\\
		 \hline
		 8 & 11 &  6 &  9 &  4 & 14 &  2 & 13\\
		 \hline
		11 & 12 &  5 &  6 & 14 & 15 &  1 &  2\\
		\hline
		15 & 14 &  2 &  1 &  8 &  7 & 10 &  9\\
		\hline
	    \end{tabular}
		\caption{A star edge coloring of $K_{8,8}$ with $15$ colors.}
	    \label{tab:K88}
	\end{figure}

\end{document}